\documentclass{amsart}

\usepackage{amssymb}
\usepackage{amsrefs}
\usepackage{overpic}
\usepackage{enumitem}   
\usepackage{graphicx} 
\usepackage{xcolor}
\usepackage[hidelinks]{hyperref}

\usepackage[pagewise]{lineno}

\graphicspath{{Figures/}} 

\newtheorem{theorem}{Theorem} 
 
\newtheorem{corollary}{Corollary}

\newtheorem{main}{Theorem} 

\begin{document}
	
\title[Limit cycles and invariant algebraic curves]{Limit cycles and invariant algebraic curves}
	
\author[Armengol Gasull and Paulo Santana]
{Armengol Gasull$^1$ and Paulo Santana$^2$}

\address{$^1$ Universitat Aut\`{o}noma de Barcelona, 08193 Bellaterra, Barcelona, Spain}
\email{armengol.gasull@uab.cat}

\address{$^2$ IBILCE--UNESP, CEP 15054--000, S. J. Rio Preto, SP, Brazil}
\email{paulo.santana@unesp.br}
	
	
	
\begin{abstract}
	We give lower bounds in terms of~$n,$ for the number of limit cycles of polynomial vector fields of degree~$n,$ having any prescribed invariant algebraic curve.  By applying them when the ovals of this curve are also algebraic limit cycles we obtain a new recurrent property for the Hilbert numbers. Finally, we apply our results to two important families of models: Kolmogorov systems and a general family of systems appearing in Game Theory. 
\end{abstract}
	
\maketitle

\section{Introduction and statement of the main results}

Let $X=(P,Q)$ be a planar polynomial vector field and $\gamma$ be one of its orbits. We say that~$\gamma$ is \emph{algebraic} if it is contained in an algebraic set. That is, if there is a polynomial $C\in\mathbb{R}[x,y]$ such that $\gamma\subset C^{-1}(\{0\})$. For simplicity, we denote the partial derivatives of a polynomial $C$ by $C_x$ and $C_y$. If we choose $C=0$ to be irreducible, then this condition reads as
\begin{equation}\label{1}
	C_x P+ C_y Q=KC.
\end{equation}
where  is $K\in\mathbb{R}[x,y]$ is known as the \emph{cofactor} of $C$.  It is easy to see that if the degree of $X$ is $n$, then the degree of $K$ is at most $n-1$.

In general, given $C\in\mathbb{R}[x,y],$ not necessarily irreducible, we say that $C=0$ is an \emph{invariant algebraic curve} for $X$ if~$C$ satisfies equation \eqref{1} for some $K\in\mathbb{R}[x,y]$ of degree at most $n-1.$ The term \emph{invariant} follows from the fact that if $C$ satisfies~\eqref{1} and the set $ C^{-1}(\{0\})$ is not empty, then it is invariant by the flow of~$X$. 

In particular, if $\gamma$ is a limit cycle  (i.e. isolated periodic orbit)  then it is said to be an \emph{algebraic limit cycle}.  If the curve  $C=0$ containing $\gamma$ is irreducible and has degree $c>0$, then we say that $\gamma$ has degree~$c$.

 For an account of known results on planar polynomial vector fields with algebraic invariant curves see~\cite[Chapter~$8$]{DumLliArt2006} and its references. The interest on this subject goes back to the seminal works of Darboux and Poincaré.
 
In this paper we study how many limit cycles the family of polynomial vector fields, with a prescribed invariant algebraic, can have. We obtain lower bounds for this maximum number of limit cycles depending on the Hilbert numbers and their known lower bounds. As a consequence, we also give a new recurrent property for the Hilbert numbers themselves. Before stating our results more precisely we recall several known properties of the Hilbert numbers.

Let $\mathcal{X}^n$ be the set of planar polynomial vector fields $X=(P,Q)$ of degree~$n,$ i.e. $\max\{\deg P,\deg Q\}=n.$ Given $X\in\mathcal{X}^n$, let $\pi(X)\in\mathbb{Z}_{\geqslant0}\cup\{\infty\}$ denote its number of limit cycles. Given $n\in\mathbb{N}$ the \emph{Hilbert number} $\mathcal{H}(n)\in\mathbb{Z}_{\geqslant0}\cup\{\infty\}$ is
\[
	{H}(n)=\sup\{\pi(X)\colon X\in\mathcal{X}^n\}.
\]
It has been named after David Hilbert due to his famous list of open problems for the $20th$ century~\cite{Browder}, being the second part of the $16th$ problem about upper bounds for $\mathcal{H}(n)$. Nowadays,  it is not even known if $\mathcal{H}(2)$ is finite. For small values of $n$, the known best lower bounds are given in Table~\ref{Table1}.

{\renewcommand{\arraystretch}{1.1}
	\begin{table}[h]
		\caption{Summary of the known lower bounds of the Hilbert numbers.}\label{Table1}
		\begin{tabular}{|c||c|c|c|c|c|c|c|c|c|c|c|}
			\hline
			$n$& $2$ &	$3$ & $4$ & $5$ & $6$ & $7$ & $8$ & $9$ & $10$ &\ldots & $n$ \\
			\hline
			$\mathcal{H}(n)$ & $4$ & $13$ & $28$ & $37$ & $53$ & $74$ &$96$ &$120$ & $142$ &\ldots& $O(n^2\ln n)$ \\
			\hline
		\end{tabular}
\end{table}}

The value of $\mathcal{H}(n)$ for $n=2$ is obtained in~\cites{ChenWang1979,Son1980}, for $n=3$ in~\cite{LiLiuYang2009}, and for $n\geqslant 4$ in~\cite{ProTor2019}. The asymptotic bound follows from the results of~\cite{ChrLlo1995,IteShu2000,Lijibin2003}. 

In the recent work~\cite{GasSan2024}, the authors proved that the Hilbert numbers are realizable by structurally stable vector fields and, as a byproduct, by hyperbolic limit cycles. 

Let us state our results. Given $C\in\mathbb{R}[x,y]$, a point $p\in\mathbb{R}^2$ is a \emph{singular point} of~$C=0$ if $C(p)=0$ and $\nabla C(p)=0$. We say that $C$ is \emph{non-degenerated} if $C=0$ has at most a finite number of singular points. If $C=0$ does not have singular points, it is \emph{non-singular}. An \emph{oval} of $C=0$ is a bounded connected component of this set without singular points of $C=0.$ Let $\mathcal{O}(C)\in\mathbb{Z}_{\geqslant0}$ denote the number of ovals of $C=0$. 

Let $\mathcal{X}_C^n\subset\mathcal{X}^n$ be the set of planar polynomial vector fields of degree $n$ having $C=0$ as an invariant algebraic curve. For $n\in\mathbb{N}$ we define the \emph{Hilbert-C number}~as
\[
	\mathcal{H}_C(n)=\sup\{\pi(X)\colon X\in\mathcal{X}_C^n\}.
\]
We say that $X\in\mathcal{X}_C^n$ is \emph{non-degenerated} if it has at most finitely many singularities at $C=0$. Let $\Sigma_C^n\subset\mathcal{X}_C^n$ denote the family of non-degenerated vector fields.
	
In our main result we use the Hilbert numbers $\mathcal{H}(n)$  to provide a lower bound for $\mathcal{H}_C$ in such way that every oval of $C=0$ is a hyperbolic limit cycle.
	
\begin{main}\label{Main1}
	Given $C\in\mathbb{R}[x,y]$ non-degenerated of degree $c>0$ and $n\in\mathbb{Z}_{\geqslant0}$, the following statements hold.
	\begin{enumerate}[label=(\alph*)]
		\item If $\mathcal{H}(n)<\infty$, then there is $X\in\Sigma_C^{n+c}$ such that every oval of $C=0$ is a limit cycle and $\pi(X)\geqslant\mathcal{H}(n)+\mathcal{O}(C),$
		\item If $\mathcal{H}(n)=\infty$, then for each $k\in\mathbb{N}$ there is $X_k\in\Sigma_C^{n+c}$ such that every oval of $C=0$ is a limit cycle and $\pi(X_k)\geqslant k+\mathcal{O}(C),$
	\end{enumerate}
being all limit cycles of $X$ or $X_k$ hyperbolic. Hence,  \begin{equation*}
	\mathcal{H}_C(n+c)\geqslant\mathcal{H}(n)+\mathcal{O}(C),\end{equation*}
and as a consequence, for $n$ big enough,  $\mathcal{H}_C(n)$  increases at least as fast as $O(n^2\ln n)$.	
\end{main}

We remark that if in Theorem~\ref{Main1} we look simply for  $X\in\mathcal{X}_C^{n+c}$ satisfying  $\pi(X)\geqslant\mathcal{H}(n)$, then the proof is much easier. Consider for instance the case $\mathcal{H}(n)<\infty.$ Take $Y\in\mathcal{X}^n$ such that $\pi(Y)=\mathcal{H}(n)$ and such that none of its limit cycles  intersects the curve $C=0.$ If we define $X=C Y\in\mathcal{X}_C^{n+c},$ then $\pi(X)=\pi(Y)$ but if $C=0$ is not empty, $X\notin \Sigma_C^{n+c},$ because the curve $C=0$ is full of singularities of $X.$

Next four results are corollaries of Theorem~\ref{Main1} and previous results on literature. The first two provide recurrent properties for the Hilbert type numbers. 

\begin{corollary}\label{Cor}
Let $L_n$ be any lower bound of $\mathcal{H}(n)$. Then $\mathcal{H}_C(n+c)\geqslant L_n+\mathcal{O}(C)$.	
\end{corollary}

From the above result, by taking the  values of Table~\ref{Table1} we obtain
\[
	\mathcal{H}_C(2+c)\geqslant4+\mathcal{O}(C), \, \mathcal{H}_C(3+c)\geqslant13+\mathcal{O}(C), \,  \mathcal{H}_C(4+c)\geqslant28+\mathcal{O}(C), \ldots
\]

\begin{corollary}\label{Cor0}
	Given $n\in\mathbb{N}$ and $m\geqslant1$ it holds
	\begin{equation}\label{3}
		\mathcal{H}(n+m)\geqslant\mathcal{H}(n)+\operatorname{Har}(m),
	\end{equation}
where
$\operatorname{Har}(m)=(m-1)(m-2)/2+\bigl[1+(-1)^m\bigr]/2.$	
Moreover the vector field of degree $n+m$ constructed to prove~\eqref{3} has $\mathcal{H}(n)+\operatorname{Har}(m)$ hyperbolic  limit cycles  and at least $\operatorname{Har}(m)$ of them are algebraic and of degree $m.$
\end{corollary}

We observe that \eqref{3} can be seen as a generalization of~\cite[Theorem~$1$]{GasSan2024}, where it was proved that $\mathcal{H}(n+1)\geqslant\mathcal{H}(n)+1$. Nevertheless notice that this last inequality is  stronger than~\eqref{3} when $m=1,$ as $\operatorname{Har}(1)=0.$ 

Next two corollaries give lower bounds for two very interesting families of models that have particular invariant algebraic curves. 

The first deals with \emph{Kolmogorov systems}. Such systems extend the classical Lotka–Volterra predator–prey systems~\cite{Karl}, and write as
\begin{equation}\label{eq4}
	\dot x=xp(x,y), \quad \dot y=yq(x,y),
\end{equation}
with $p$, $q\in\mathbb{R}[x,y].$  Notice that for them, if we consider
$C(x,y)=xy,$ then $C=0$ is a reducible invariant algebraic curve. In this case we write $\mathcal{H}_{C}$ as~$\mathcal{H}_K$. 

 Since $\deg C=2$ and $xy=0$ has no ovals, by Theorem~\ref{Main1}, $\mathcal{H}_K(n)\geqslant\mathcal{H}(n-2).$ In a first version of this work we had used this inequality for $n\geqslant5.$ 
In a recent 2025 meeting the authors were communicated by Plesa that the technique employed by him in~\cite{Plesa} could improve that lower bound to $\mathcal{H}_K(n)\geqslant\mathcal{H}(n-1)$. We have used this technique  in next two results. We greatly appreciate his suggestion. As we will see in the proof of next corollaries, this improvement can be done due to the very particular and simple shape of the invariant algebraic curves in these cases.

\begin{corollary}\label{Cor1}  
	Let $\mathcal{H}_K(n)$ denote the maximum number of limit cycles of~\eqref{eq4} in the first quadrant, where $p$ and $q$ vary among all polynomials of degree at most $n-1$. Then $\mathcal{H}_K(n)\geqslant\mathcal{H}(n-1)$ and  next table gives their exact values  for $n=1,2$ and some  lower bounds for $n\geqslant3$.
{\renewcommand{\arraystretch}{1.1}
	\begin{table}[h]
				\begin{tabular}{|c||c|c|c|c|c|c|c|c|c|c|}
			\hline
			$n$& $1$& $2$&	$3$ & $4$ & $5$ & $6$ & $7$ & $8$ & \ldots &  $n$\\
			\hline
			$\mathcal{H}_K(n)$ &$0$& $0$& $6$ & $13$ & $28$ & $37$ & $53$ & $74$ & \ldots & $O(n^2\ln n)$\\
			\hline
		\end{tabular}
\end{table}}
\end{corollary}

  In case $n=1$ the result is simply because linear systems do not have limit cycles. When $n\in\{2,3\}$, it follows from previous works, see the references in the proof. For the case $n=3$  the value was obtained by a high order Andronov--Hopf bifurcation. By using this last method  it is proved in~\cite{CarCruzGou2023}  that  $\mathcal{H}_K(4)\geqslant13$ and $\mathcal{H}_K(5)\geqslant22.$  Our approach recovers the same lower bound when $n=4$ and improves it  when  $n=5.$  As far as we know, the results for $n\geqslant6$ and  the asymptotic lower bound for large~$n$ are new.
  Moreover, from our proof we also obtain that all existing configurations  of hyperbolic limit cycles for polynomial  systems of degree~$n-1$ are realizable in the first quadrant  for Kolmogorov systems of degree $n.$  From this point of view the above two results showing that $\mathcal{H}_K(4)\geqslant13$ are different. While the 13 limit cycles given in~\cite{CarCruzGou2023} are nested, the 13 given by our approach are not nested. Indeed, until now only 12 nested limit cycles have been found for cubic systems.

By using the same tools we obtain a similar result for some models in  Game Theory. For our goals it is sufficient to observe that when the game has two players it can modeled by vector fields of the form
\begin{equation}\label{eq5}
	\dot x=x(x-1)p(x,y), \quad \dot y=y(y-1)q(x,y),
\end{equation}
with $p$ and $q\in\mathbb{R}[x,y].$  For details about their construction, and the interpretation of $x$ and $y$ as strategies, see~\cite{GasGouSan2025, HS} and their references. Clearly, these models have the reducible invariant algebraic curve
$C(x,y)=x(x-1)y(y-1)=0$ 
of degree $c=4.$ Moreover, the region of interest for applications is the open square~$\mathcal{Q},$ delimited by the invariant four straight lines, which is also clearly invariant by the flow. To recall the shape of this region, in this case we write   $\mathcal{H}_C=\mathcal{H}_{{\scriptscriptstyle {\square}}}$. 

\begin{corollary}\label{Cor2}  
	
	Let $\mathcal{H}_{\scriptscriptstyle {\square}}(n)$ denote the maximum number of limit cycles of~\eqref{eq5} in the open square $\mathcal{Q}$, where $p$ and $q$ vary among all polynomials of degree at most $n-2$. Then $\mathcal{H}_{\scriptscriptstyle {\square}}(n)\geqslant\mathcal{H}(n-2)$ and  next table gives their exact values for $n=2,3$ and some  lower bounds for $n\geqslant4$.
		{\renewcommand{\arraystretch}{1.1}
		\begin{table}[h]
			\begin{tabular}{|c||c|c|c|c|c|c|c|c|c|c|c|}
				\hline
				$n$&  $2$&	$3$ & $4$ & $5$ & $6$ & $7$ & $8$ & $9$ & $10$& \ldots &  $n$\\
				\hline
				$\mathcal{H}_{\scriptscriptstyle {\square}}(n)$ & $0$& $1$ & $5$ & $13$ & $28$ & $37$ & $53$ & $74$ & $96$& \ldots & $O(n^2\ln n)$\\
				\hline
			\end{tabular}
	\end{table}}
\end{corollary}

 Cases $2\leqslant n \leqslant 4$ follow from previous known results, see the references in the proof. Again the main tool when $n=4$ was the study of a high order Andronov--Hopf bifurcation.  For $n\geqslant5,$ as well as for the asymptotic lower bound for big $n$, the results are new. Similarly that in the Kolmogorov case,  all existing configurations  of hyperbolic limit cycles for polynomial systems of degree~$n-2$ are realizable in~$\mathcal{Q}$  for systems of degree $n$ of the form~\eqref{eq5}.

The paper is organized as follows. In Section~\ref{Sec2} we give some preliminary results about properties of the Hilbert numbers and about prescribed algebraic limit cycles. The main theorem and its corollaries are proved in Section~\ref{Sec3}. 

\section{Preliminaries}\label{Sec2}

The first result that we will use is a simplified version of~\cite[Theorem~$2$]{GasSan2024}.

\begin{theorem}[\cite{GasSan2024}]\label{T1}
	For $n\in\mathbb{N}$, the following statements hold.
	\begin{enumerate}[label=(\alph*)]
		\item If $\mathcal{H}(n)<\infty$, then there is $X\in\mathcal{X}^n$ having $\mathcal{H}(n)$ hyperbolic limit cycles.
		\item If $\mathcal{H}(n)=\infty$, then for each $k\in\mathbb{N}$ there is $X_k\in\mathcal{X}^n$ having at least~$k$ hyperbolic limit cycles.
	\end{enumerate}
\end{theorem}

We will also need an adaptation of~\cite[Theorem 1]{Chr2001}. Although its proof is essentially the same, for sake of self-completeness we include it here.

\begin{theorem}\label{T2}
	Let $C\in\mathbb{R}[x,y]$, be non-degenerated and of degree $c>0$,  $D\in\mathbb{R}[x,y]$ a polynomial of degree one that does not divide $C$ and such that the $D=0$ does not intersect any oval of $C=0$. Let $\alpha$, $\beta\in\mathbb{R}$ such that $\alpha D_x+\beta D_y\neq0$. Then the polynomial vector field $Y=(P,Q)$ given by
	\begin{equation}\label{6}
		P=\alpha C-DC_y, \quad Q=\beta C+DC_x,
	\end{equation}
	has all the ovals of $C=0$ as hyperbolic limit cycles and $Y\in\Sigma_C^{c}$.
\end{theorem}

\begin{proof}
	It is clear that $C=0$ is an invariant algebraic curve for \eqref{6} because
	$C_xP+C_yQ=\big(\alpha C_x+\beta C_y\big)C.$  Moreover, clearly  $Y\in\Sigma_C^{c}$.	Let $\gamma$ be an oval of $C=0$, i.e. a bounded connected component of $C=0$ without singular points.  Since $D=0$ does not intersect~$\gamma$ and $\nabla C$ never vanishes at it, we have that $Y$ does not have singularities at $\gamma$ and thus it is a periodic orbit of $Y$. It can be parameterized by time and let~$T$ denote its minimal period. 
	To prove that it is a hyperbolic limit cycle it suffices to show that next integral does not vanishes, see~\cite[Section 7.3]{DumLliArt2006}. We have,
	\begin{align*}
		\displaystyle \int_0^{T}\operatorname{div} (Y)\,{\rm d} t &=	\displaystyle \int_0^{T}P_x+Q_y\,{\rm d} t=\int_0^{T}(\alpha+D_y)C_x+(\beta-D_x)C_y\;{\rm d} t \vspace{0.2cm} \\
		&\displaystyle\stackrel{(*)}{=}\int_0^{T}\frac{1}{D}\bigl[(\alpha+D_y)\underbrace{(\beta C+DC_x)}_{{\rm d} y/{\rm d} t}-(\beta-D_x)\underbrace{(\alpha C-D C_y)}_{{\rm d} x/{\rm d} t}\bigr]{\rm d} t \vspace{0.2cm} \nonumber\\
		\displaystyle&=\int_{\gamma}\frac{\alpha+D_y}{D}{\rm d} y-\frac{\beta-D_x}{D}{\rm d} x=\pm\int_{\Gamma}\frac{\alpha D_x+\beta D_y}{D^2}\;{\rm d} x{\rm d} y\neq0.	\nonumber		
	\end{align*}	
	Notice that in the last equality we used Green's theorem, $\Gamma$ denotes the interior region bounded by $\gamma,$ the sign in front of the last integral depends on the time orientation of $\gamma$ and in $(*)$ we have used that in the integrated function we can take $C=0$ because $\gamma$ is contained in this set.
\end{proof}

\section{Proofs}\label{Sec3}

\begin{proof}[Proof of Theorem~\ref{Main1}]	
	Suppose first $\mathcal{H}(n)<\infty$.  It follows from Theorem~\ref{T1}$(a)$ that there is $X=(U,V)$ such that $\pi(X)=\mathcal{H}(n)$ and all its limit cycles are hyperbolic. Let $B\subset\mathbb{R}^2$ be a closed ball containing all the limit cycles introduced so far (i.e. all the limit cycles of $X$) in its interior. Translating and scaling $X$ if necessary, we can suppose $B\cap C^{-1}(\{0\})=\emptyset$.
	
	It follows from Theorem~\ref{T2} that there is $Y=(P,Q)\in\Sigma_C^c$ having every oval of~$C$ as hyperbolic limit cycle. Recall that then $C_xP+C_yQ=KC,$ for some cofactor~$K.$ Given $\varepsilon>0$ let $Z=(R,S)$ be given by
	\begin{equation*}
		R=CU+\varepsilon P, \quad S=CV+\varepsilon Q.
	\end{equation*}
	Notice that $C_xR+C_yS=\big(C_xU+C_yV+\varepsilon K\big)C$ and thus $Z\in\mathcal{X}_C^{n+c}$. By Theorem~\ref{T2}, if $\gamma$ is one of the limit cycles of $Y$, given by one of the ovals of $C,$ it is hyperbolic. Therefore, for $\varepsilon>0$ small enough it is also a  hyperbolic limit cycle for  $Z.$
	Since all the limit cycles of $X$ are inside $B,$ are also hyperbolic, and $B\cap C^{-1}(\{0\})=\emptyset$, it follows that they persist and give rise to nearby hyperbolic limit cycles for $Z$ when $\varepsilon>0$ is again small enough. Hence there is $\varepsilon>0$ small enough  such all the algebraic limit cycle of $Y$ remain as  limit cycles of~$Z$ and the hyperbolic limit cycles of $X$ also persist as limit cycles of~$Z$, being small pertubations of the original ones,  proving  
	\begin{equation}\label{eq:bd}
		\mathcal{H}_C(n+c)\geqslant\pi(Z)\geqslant\pi(X)+\pi(Y)=\mathcal{H}(n)+\mathcal{O}(C).
	\end{equation}
	Observe now that $Z|_{C=0}=\varepsilon Y|_{C=0}$ and thus $Z\in\Sigma_C^{n+c}$.
	
	If $\mathcal{H}(n)=\infty$, then we have from Theorem~\ref{T1}$(b)$ that for each $k\in\mathbb{N}$ there is $X_k\in\Sigma^{n}$ with at least $k$ hyperbolic limit cycles. Applying the previous reasoning on each $X_k$ we obtain a sequence $\{Z_k\}\subset\Sigma_C^{n+c}$, such that $\pi(Z_k)\geqslant k+\mathcal{O}(C)$. 
	
	It follows from~\cite[Theorem~$2.1$]{IteShu2000} that there is $M>0$ such that for each $n\geqslant3$,
	$
		\mathcal{H}(n)\geqslant (n^2\log_2n)/2-Mn^2\log_2\log_2n.$
		Hence from~\eqref{eq:bd} we have that, for $n$ big enough, $\mathcal{H}_C(n)$ increases at least as fast as $O(n^2\ln n)$.
\end{proof}

\begin{proof}[Proof of Corollary~\ref{Cor}]
	From~\eqref{eq:bd} we have $\mathcal{H}_C(n+c)\geqslant\mathcal{H}(n)+\mathcal{O}(C)\geqslant L_n+\mathcal{O}(C).$
\end{proof}	

\begin{proof}[Proof of Corollary~\ref{Cor0}] 
	For each $m\in\mathbb{N}$ we have from \emph{Harnack’s Curve Theorem}~\cite{Wilson} that there is $C\in\mathbb{R}[x,y],$  with $C=0$ irreducible and of degree $m>0,$ with $\operatorname{Har}(m)$ ovals. The proof follows from Theorem~\ref{Main1} applied to this curve.
\end{proof}

\begin{proof}[Proof of Corollary~\ref{Cor1}]  
	Case $n=1$ follows from the fact that linear systems do not have limit cycles. For $n=2$ it was proved by Bautin~\cite{Bautin} that quadratic Kolmogorov vector fields cannot have limit cycles. See also~\cite{Coppel}*{Section~$5$}. By using~\cite{Lloyd} we get $\mathcal{H}_K(3)\geqslant 6$. From Theorem~\ref{Main1} we obtain the asymptotic expansion for~$n$ large.
	
	To end the  proof it suffices to prove $\mathcal{H}_K(n)\geqslant\mathcal{H}(n-1)$, $n\geqslant1,$ and use the known lower bounds of $\mathcal{H}(n-1)$ for $n\geqslant5$ given in Table~\ref{Table1}. We adapt to our setting the technique developed by Plesa~~\cite{Plesa}. To this end, similarly to the proof of Theorem~\ref{Main1} suppose $\mathcal{H}(n-1)<\infty$ and consider a vector field $X=(U,V)$ of degree $n-1$ such that $\pi(X)=\mathcal{H}(n-1)$ and all its limit cycles hyperbolic, which can be supposed by Theorem~\ref{T1}. Let $B\subset\mathbb{R}^2$ be a closed ball containing all the limit cycles of $X$ in its interior. Consider the family of vector fields of degree $n,$ $X_\mu=(R_\mu,S_\mu),$  given by
	\[
		R_\mu=(1+\mu x)U, \quad S_\mu=(1+\mu y)V,
	\]
	where $\mu$ is a non-negative parameter. Notice that $X_0=X,$ and the following important properties of $X_\mu,$ when $\mu>0$ is small enough:
	\begin{itemize}
		\item It has the invariant straight lines $x=-1/\mu$, and $y=-1/\mu$. Moreover, since $B$ is compact, such lines do not intersect $B.$
		\item  Restricted to $B,$ and again due to its compactness,  $X_\mu$  can be seen as an arbitrary small perturbation of $X_0=X.$ This is no more true far away from~$B.$
	\end{itemize}
As a consequence, 	since all the limit cycles of $X$ are hyperbolic and they are contained in $B$ it follows that if $\mu>0$ is small enough, all them persist for $X_\mu.$  Finally, if we consider the new coordinates $x_1=1+\mu x$, $y_1=1+\mu y$, then $X_\mu$ becomes a Kolmogorov vector field of the form~\eqref{eq4}, of  degree $n$ and with at least $\mathcal{H}(n-1)$ hyperbolic limit cycles in the first quadrant and keep their configuration. The case $\mathcal{H}(n-1)=\infty$ follows similarly to Theorem~\ref{Main1}.
\end{proof}

\begin{proof}[Proof of Corollary~\ref{Cor2}] 
	When $n=2$ the system is easily integrable and $\mathcal{H}_{\scriptscriptstyle {\square}}(2)=0$. Case $n=3$ follows from~\cite[Theorem~$2.5.4$]{Kooij}. From~\cites{GasGouSan2025,GraSza2024} we obtain $\mathcal{H}_{\scriptscriptstyle {\square}}(4)\geqslant5$. As in Corollary~\ref{Cor1}, from Theorem~\ref{Main1} we get the asymptotic expansion for $n$ large. 				
				
	To end the  proof it suffices to prove that $\mathcal{H}_{\scriptscriptstyle {\square}}(n)\geqslant\mathcal{H}(n-2)$, $n\geqslant2,$ and use the lower bounds of $\mathcal{H}(n-1)$ for $n\geqslant5$ of Table~\ref{Table1}. The proof of this inequality follows by adapting the proof of the similar one used in  Corollary~\ref{Cor1}. We only point  out  some minor differences. In this case, when  $\mathcal{H}(n-2)<\infty$ we start with a vector field $X=(U,V)$ of degree $n-2$ such that $\pi(X)=\mathcal{H}(n-2)$ and all its limit cycles hyperbolic, by Theorem~\ref{T1}. Then, the suitable 1-parameter family of vector fields $X_\mu=(R_\mu,S_\mu)$ of degree $n$ that allows us to prove the result is 
	\[
		R_\mu =(1+\mu x)(1-\mu x)U, \quad S_\mu =(1+\mu y)(1-\mu y)V.
	\]
	The case $\mathcal{H}(n-2)=\infty$ also follows similarly.	We skip the details.
\end{proof}

\section*{Acknowledgments}
This work is supported by the Spanish State Research Agency, through the project PID2022-136613NB-I00, and by S\~ao Paulo Research Foundation (FAPESP), Brazil, grants 2021/01799-9 and 2024/15612-6.

\end{document}